\documentclass[12pt]{amsart}%
\usepackage{amssymb}
\usepackage{amsmath}
\usepackage{amsfonts}
\usepackage{graphicx}%
\newtheorem{theorem}{Theorem}[section]

\newtheorem{corollary}[theorem]{Corollary}

\newtheorem{lemma}[theorem]{Lemma}

\newtheorem{proposition}[theorem]{Proposition}

\newtheorem{result}[theorem]{Result}

\begin{document}
\title{Mitsch's order and inclusion for binary relations and partitions}
\author{D. G. FitzGerald}
\address{School of Mathematics and Physics, University of Tasmania\\
Private Bag 37, Hobart 7001, Australia.}
\email{D.FitzGerald@utas.edu.au}

\begin{abstract}
Mitsch's natural partial order on the semigroup of binary relations is here
characterised by equations in the theory of relation algebras. The natural
partial order has a complex relationship with the compatible partial order of
inclusion, which is explored by means of a sublattice of the lattice of
preorders on the semigroup. The corresponding sublattice for the partition
monoid is also described.

\end{abstract}
\maketitle

\section{ Natural partial orders}

In \cite{Mi}, Heinz Mitsch formulated a characterisation of the natural
partial order $\leq$ on the full transformation semigroup $\mathcal{T}_{X}$
which did not use inverses or idempotents, and went on to define the natural
partial order $\leq$ on any semigroup $S$ by
\begin{equation}
a\leq b\text{ if }a=b\text{ or there are }x,y\in S\text{ such that}~a=ax=bx=yb
\label{def}%
\end{equation}
for $a,b\in S$. (The discovery was also made independently by P. M. Higgins,
but remained unpublished.) Observe that $a=ya$ follows.\emph{ }Mitsch's
natural partial order has now been characterised, and its properties
investigated, for several concrete classes of non-regular semigroups---in
\cite{MSS, SS} for some semigroups of (partial) transformations, and by Namnak
and Preechasilp \cite{NP} for the semigroup $\mathcal{B}_{X}$ of all binary
relations on the set $X$.

The partial order of inclusion which is carried by $\mathcal{B}_{X}$ may also
be thought of as `natural', and it is the broad purpose of this note to
discuss the relationship between these two partial orders on $\mathcal{B}%
_{X}.$ Moreover, the same questions are addressed for the partition monoid
$\mathcal{P}_{X},$ which also carries two `natural' partial orders. So we
shall use a slightly different nomenclature here for the sake of clarity,
mostly referring to partial orders as just \emph{orders, }and the natural
partial order as \emph{Mitsch's order}. We begin by collecting some
information about $\mathcal{B}_{X}.$

\section{Binary relations}

The notation used here for binary relations follows that found in, for
example, Clifford and Preston \cite{CP}, with the addition of complementation
of relations defined by
\[
x\,\alpha^{c}\,y\iff\left(  x,y\right)  \not \in \alpha
\]
for $x,y\in X.$ Note that the symbol $\circ$ for composition will be
suppressed, except for the composites of order relations on $\mathcal{B}_{X}.$
We will make use of the identity relation on $X,$ $\iota=\left\{  \left(
x,x\right)  :x\in X\right\}  ,$ and the universal relation $\omega=X\times X.$

The following logical equivalence will also be required; it is the `Theorem
K'\ of De Morgan \cite[p. \textit{xxx}]{DeM} (see also e.g. \cite{Mdx}).

\begin{result}
\label{morgan}For $\alpha,\beta,\xi\in\mathcal{B}_{X},~$%
\[
\beta\xi\subseteq\alpha\iff\xi\subseteq\left(  \beta^{-1}\alpha^{c}\right)
^{c}\iff\beta\subseteq\left(  \alpha^{c}\xi^{-1}\right)  ^{c}.
\]

\end{result}

Result \ref{morgan} will be used in the following form:

\begin{corollary}
\label{coroll}(i) If the set $\left\{  \xi\in\mathcal{B}_{X}:\beta\xi
=\alpha\right\}  $ is non-empty, it has greatest element $\left(  \beta
^{-1}\alpha^{c}\right)  ^{c}$ in the inclusion order.\newline\hspace
{2.75cm}(ii) If the set $\left\{  \xi\in\mathcal{B}_{X}:\xi\beta
=\alpha\right\}  $ is non-empty, it has greatest element $\left(  \alpha
^{c}\beta^{-1}\right)  ^{c}$ in the inclusion order.
\end{corollary}

\begin{proof}
(i) By Result \ref{morgan} above, $\beta\xi=\alpha$ implies $\xi
\subseteq\left(  \beta^{-1}\alpha^{c}\right)  ^{c}$; but the latter implies,
again using Result \ref{morgan}, that
\[
\alpha=\beta\xi\subseteq\beta\left(  \beta^{-1}\alpha^{c}\right)
^{c}\subseteq\alpha.
\]
Part (ii) is proven dually.
\end{proof}

\section{Equational criterion for the Mitsch order on $\mathcal{B}_{X}$}

The basic definition (\ref{def}) of $\alpha\leq\beta$ is existentially
quantified. In $\mathcal{B}_{X}$ there is a purely equational equivalent:

\begin{theorem}
\label{criterion}For $\alpha,\beta\in\mathcal{B}_{X},$ $\alpha\leq\beta$ if
and only if
\[
\alpha=\alpha\left(  \beta^{-1}\alpha^{c}\right)  ^{c}=\beta\left(  \beta
^{-1}\alpha^{c}\right)  ^{c}=\left(  \alpha^{c}\beta^{-1}\right)  ^{c}\beta.
\]

\end{theorem}

\begin{proof}
Suppose $\alpha\leq\beta.$ By definition there are $\xi,\eta$ such that
\[
\alpha=\alpha\xi=\beta\xi=\eta\beta
\]
and hence both
\[
\eta\alpha=\eta\beta\xi=\alpha\xi=\alpha
\]
and%
\[
\alpha=\beta\left(  \beta^{-1}\alpha^{c}\right)  ^{c}=\left(  \alpha^{c}%
\beta^{-1}\right)  ^{c}\beta,
\]
the latter by Corollary \ref{coroll}. Now%
\begin{align*}
\alpha &  =\eta\alpha=\eta\beta\left(  \beta^{-1}\alpha^{c}\right)
^{c}=\alpha\left(  \beta^{-1}\alpha^{c}\right)  ^{c}\\
&  =\beta\left(  \beta^{-1}\alpha^{c}\right)  ^{c}=\left(  \alpha^{c}%
\beta^{-1}\right)  ^{c}\beta.
\end{align*}
Conversely,
\[
\alpha=\alpha\left(  \beta^{-1}\alpha^{c}\right)  ^{c}=\beta\left(  \beta
^{-1}\alpha^{c}\right)  ^{c}=\left(  \alpha^{c}\beta^{-1}\right)  ^{c}\beta
\]
demonstrates $\alpha\leq\beta.$\hfill
\end{proof}

\medskip Regarded as a computation, this criterion is of polynomial time
complexity in $\left\vert X\right\vert $, as is also the case for tests of the
divisibility preorders (by Corollary \ref{coroll}, $\alpha=\beta\xi$ if and
only if $\alpha=\beta\left(  \beta^{-1}\alpha^{c}\right)  ^{c},$ etc.), but in
contrast to the NP-complete tests for the $\mathcal{J}$-preorder \cite{My}. Of
course algorithmic complexity is not the only issue. Namnak and Preechasilp
\cite{NP} characterise Mitsch's order for binary relations with the aid of
Zaretski\u{\i}'s criteria for divisibility \cite{Z} which, though also of
worst-case exponential complexity, proved very convenient for the purposes of
finding compatible elements, atoms and maximal elements in the Mitsch order
\cite{NP}. The equations of Theorem \ref{criterion} are complex in the
different sense that they belong to a theory of semigroups enriched by
operations of inversion and complementation, in fact, to the theory of
relation algebras \cite{Mdx}. It may be observed that all the proofs of
sections 4 and 5 apply to the multiplicative reducts of (abstract) relation
algebras \cite{Mdx}, and not just the representable ones $\mathcal{B}_{X}.$

\section{Connexions with the inclusion order}

In discussing Mitsch's order and the inclusion order on $\mathcal{B}_{X}$,
\cite{NP} notes their logical independence. In fact, we can see that every
inclusion atom is a Mitsch atom, and every non-empty relation on $X$ is in an
inclusion interval between two Mitsch atoms. Moreover, permutations of $X$ are
Mitsch-maximal, but far from being either maximal elements or atoms in the
inclusion order. Yet the relationship between the two orders is subtle, and
worthy of further exploration. We illustrate this by next finding a
substructure of $\mathcal{B}_{X}$ in which $\leq$ agrees with $\subseteq~,$
and others where $\leq$ agrees with reverse inclusion $\supseteq\,.$ We shall
use the following statement, easily proved.

\begin{lemma}
\label{hall}Let $S$ be a semigroup, $T$ be a regular subsemigroup of $S,$ and
$a,b\in T.$ Then $a\leq b$ in $T$ if and only if $a\leq b$ in $S.$
\end{lemma}

The symmetric inverse monoid $\mathcal{I}_{X}$ is a regular subsemigroup of
$\mathcal{B}_{X},$ and the natural partial order on $\mathcal{I}_{X}$
coincides with inclusion. So we may apply Lemma \ref{hall} to $\mathcal{I}%
_{X}.$

\begin{corollary}
\label{ix}(i) If $\alpha,\beta\in\mathcal{I}_{X},~$then $\alpha\leq\beta~$if
and only if $\alpha\subseteq\beta.$\newline\hspace*{2.75cm}(ii) If $\beta
\in\mathcal{I}_{X},$ then $\alpha\subseteq\beta$ implies $\alpha\leq\beta.$
\end{corollary}

Part (ii) relies on the observation that $\beta\in\mathcal{I}_{X}$ and
$\alpha\subseteq\beta$ imply $\alpha\in\mathcal{I}_{X}$. Of course,
$\alpha\leq\beta$ need not entail $\alpha\in\mathcal{I}_{X}.$

To obtain pairs $\left(  \alpha,\beta\right)  $ such that $\leq$ agrees with
reverse inclusion $\supseteq,$ first let $\alpha$ be a reflexive and
transitive relation on $X$ (so, a preorder), and define a subset of
$\mathcal{B}_{X}$ by
\[
F\left(  \alpha\right)  =\left\{  \beta\in\mathcal{B}_{X}:\alpha\beta
=\alpha=\beta\alpha\right\}  .
\]

\begin{proposition}
\label{fa} For all $\beta\in F\left(  \alpha\right)  ,$ $\alpha\leq\beta~$if
and only if~$\beta\subseteq\alpha.$
\end{proposition}

\begin{proof}
Clearly $F\left(  \alpha\right)  $ is a subsemigroup of $\mathcal{B}_{X},$
with zero element $\alpha.$ So for all $\beta\in F\left(  \alpha\right)  ,$ we
have $\alpha\leq\beta$ and, since $\iota\subseteq\alpha,$ also $\beta
\subseteq\beta\alpha=\alpha.$
\end{proof}

Corollary \ref{ix} shows that it may be instructive to consider the
conjunction of the natural partial order with set inclusion, which is an order
on $\mathcal{B}_{X}$ which we naturally write as $\leq\cap\subseteq\,.$
Similarly, Proposition \ref{fa} suggests investigating the conjuction of
$\leq$ with reverse inclusion, an order written as $\leq\cap\supseteq\,$. The
next Proposition, which extends Proposition \ref{fa}, gives an alternative
characterisation for $\leq\cap\supseteq~;$ there seems to be no analogous
description of $\leq\cap\subseteq$ .

\begin{proposition}
\label{meet}For all $\alpha,\beta\in\mathcal{B}_{X},~\alpha\leq\beta$ and
$\alpha\supseteq\beta$ if and only if there are $\varepsilon=\varepsilon^{2}$
and $\phi=\phi^{2}$ such that $\iota\subseteq\varepsilon$ and $\alpha
=\beta\varepsilon=\phi\beta.$
\end{proposition}

\begin{proof}
Let $\alpha\leq\beta.$ From Corollary \ref{coroll} there exists the element
$\varepsilon,$ maximum with respect to $\subseteq$ such that $\alpha
=\beta\varepsilon,$ and $\alpha=\alpha\varepsilon$ also holds by Theorem
\ref{criterion}. Then also $\alpha=\beta\varepsilon^{2}$ and thus
$\varepsilon^{2}\subseteq\varepsilon.$ But from $\beta\iota\subseteq\alpha$ we
have $\iota\subseteq\varepsilon,$ so $\varepsilon\subseteq\varepsilon^{2}$ and
$\varepsilon=\varepsilon^{2}.$ Similarly $\phi=\phi^{2}$ with $\alpha
=\phi\beta.$

Conversely, if the conditions hold then plainly $\alpha\leq\beta$ as in the
regular case, but also $\iota\subseteq\varepsilon$ implies $\beta
\subseteq\varepsilon\beta=\alpha.$
\end{proof}

\section{The sublattice of preorders generated by $\leq,$ $\subseteq$ and
$\supseteq$}

The subsemigroups $F\!\left(  \alpha\right)  $ and $\mathcal{I}_{X}$ of
section 4 clearly show that, if $\left\vert X\right\vert \geq2,$ there can be
no order relation on $\mathcal{B}_{X}$ which contains $\leq$ and either of
$\subseteq$ and $\supseteq$\thinspace. However the set of preorders on
$\mathcal{B}_{X}$ is bounded by the relation of equality $=$ and the universal
relation on $\mathcal{B}_{X},$ and is closed under arbitrary intersections,
and so forms a lattice. In a similar situation, the papers \cite{MSS} and
\cite{SS} derive interesting results from considering the composite of
inclusion with $\leq$, and this idea also turns out to be useful here.

\begin{proposition}
\label{rectang}For all $\alpha,\beta\in\mathcal{B}_{X},$ there exists
$\gamma\in\mathcal{B}_{X}$ such that $\alpha\subseteq\gamma\leq\beta$ if and
only if $\alpha\omega\alpha\subseteq\beta\omega\beta.$
\end{proposition}

\begin{proof}
Suppose $\alpha\subseteq\gamma\leq\beta.$ Then there are $\xi,\eta$ such that
$\gamma=\beta\xi=\eta\beta,$ so
\[
\alpha\omega\alpha\subseteq\gamma\omega\gamma=\beta\xi\omega\eta\beta
\subseteq\beta\omega\beta.
\]
Since $\beta\subseteq\beta\beta^{-1}\beta\subseteq\beta\omega\beta,$ we have
\[
\beta\omega\beta\subseteq\beta\omega\beta\omega\beta\subseteq\beta\omega
\beta.
\]
But now $\beta\omega\beta=\beta\omega\beta\omega\beta$ shows that $\beta
\omega\beta\leq\beta.$ As above we have $\alpha\subseteq\alpha\omega\alpha,$
so $\alpha\omega\alpha\subseteq\beta\omega\beta$ implies $\alpha\subseteq
\beta\omega\beta\leq\beta.$
\end{proof}

As a corollary, we have the join of the Mitsch and inclusion orders in the
lattice of preorders on $\mathcal{B}_{X}.$

\begin{corollary}
\label{lteinc}(i) The composite $\leq\circ\subseteq$ ~is contained in
$\subseteq\circ\leq~.$ \newline\hspace*{2.4cm}(ii) $\subseteq\circ\leq$ is the
join of $\subseteq$ and $\leq$ in the lattice of preorders on $\mathcal{B}%
_{X}.$
\end{corollary}

\begin{proof}
(i) First let us note that $\leq$ and $\subseteq$ are subsets of
$\subseteq\circ\leq~.$ It is clear from Proposition \ref{rectang} that
$\subseteq\circ\leq$ is transitive, i.e.,
\[
\left(  \subseteq\circ\leq\right)  \circ\left(  \subseteq\circ\leq\right)
=\left(  \subseteq\circ\leq\right)
\]
and it follows that $\left(  \leq\circ\subseteq\right)  $ is contained in
$\left(  \subseteq\circ\leq\right)  .$

(ii) Also it is immediate that $\subseteq\circ\leq$ is reflexive, and so it is
a preorder on $\mathcal{B}_{X}.$ Any preorder on $\mathcal{B}_{X}$ containing
both $\subseteq$ and $\leq$ also contains $\subseteq\circ\leq\,.$ Hence
$\subseteq\circ\leq$ is the join of $\subseteq$ and $\leq\,.$
\end{proof}

That the containment in (i) is proper (for $\left\vert X\right\vert \geq2$) is
shown by the following instance. Consider a pair of distinct permutations
$\pi,\rho$: we have $\pi\subseteq\omega\leq\rho,$ but $\pi\leq\alpha
\subseteq\rho$ implies $\pi=\alpha=\rho,$ a contradiction. It also follows
that $\leq\circ\subseteq~$is not a preorder.

We turn to the composites and join of $\leq$ with reverse inclusion.

\begin{proposition}
\label{cnilte}The composite $\supseteq\circ\leq$ is the universal relation on
$\mathcal{B}_{X}$ and the join of $\supseteq$ and $\leq$ in the lattice of
preorders on $\mathcal{B}_{X}.$
\end{proposition}

\begin{proof}
For any $\alpha,\beta\in\mathcal{B}_{X},$ $\alpha\supseteq\varnothing\leq
\beta,$ so $\mathcal{B}_{X}\times\mathcal{B}_{X}$ coincides with
$\supseteq\circ\leq$ . Now $\mathcal{B}_{X}\times\mathcal{B}_{X}$ is plainly a
preorder, and any preorder containing both $\leq$ and $\supseteq$ must contain
$\supseteq\circ\leq$ and hence $\mathcal{B}_{X}\times\mathcal{B}_{X}.$
\end{proof}

Here too, we see that the reverse composite $\leq\circ\supseteq$ is properly
contained in $\supseteq\circ\leq\,=\,\mathcal{B}_{X}\times\mathcal{B}_{X}$
when $\left\vert X\right\vert \geq2$ (and so is not a preorder), since
$\iota\leq\alpha\supseteq\omega$ implies both $\alpha=\iota$ and
$\alpha=\omega.$

Thus we are able to describe the relationships between $\leq~,$ $\subseteq$
and $\supseteq$ in terms of the sublattice generated by $\leq\,,\subseteq$ and
$\supseteq\,$within the lattice of preorders on $\mathcal{B}_{X}$ $.$ This
sublattice is summarised by a Hasse diagram in Fig. 1; filled circles denote
orders, and additional labels in parentheses summarise conditions for
$\alpha,\beta$ to be related by the preorder.%

\begin{center}
\includegraphics[
trim=0.000000in 0.000000in -0.000367in -0.002311in,
natheight=3.300900in,
natwidth=3.672000in,
height=7.0348cm,
width=7.8129cm
]%
{Fig1.tif}%
\label{F1}%
\end{center}

\begin{center}
{\small Fig. 1. A sublattice of preorders on }$\mathcal{B}_{X}.$%
{\small ~Filled circles denote orders.}
\end{center}

\bigskip

\section{Dualising: partition monoids}

The partition monoid on the set $X,$ denoted by $\mathcal{P}_{X},$ generates
the partition algebra, which is important in group representation theory and
statistical mechanics; for expositions, see \cite{Mn} and \cite{HR}. Study of
$\mathcal{P}_{X}$ \textit{qua }semigroup is more recent and scanty. The reader
is referred to the articles \cite{W}, \cite{E}, \cite{FL}, which include full
descriptions and examples; here we give just a concise description, which
suffices for the present purpose. A \emph{partition over }$X$ is a quotient
object (or equivalently, a partition) of the coproduct (disjoint union)
$X\sqcup X$ of two copies of the set $X.$ It is often convenient to represent
$a\in\mathcal{P}_{X}$ as a graph on the vertex set $X_{0}\cup X_{1}$ (where
$X_{0}$ and $X_{1}$ are disjoint copies of $X$) in which, for $i,j$ $\in
X_{0}\cup X_{1},~\left\{  i,j\right\}  $ is an edge if and only if $i,j$
belong to the same block of the partition $a.$ (The resulting graph is then a
union of cliques.) If a block of $a\in\mathcal{P}_{X}$ has elements in both
$X_{0}$ and $X_{1},$ we say it is a \emph{transversal} block.

For the product in $\mathcal{P}_{X},$ let $b$ be represented as a union of
cliques on $X_{1}\cup X_{2}$ (where $X_{2}$ is another copy of $X$ disjoint
from both $X_{0}$ and $X_{1}$). Thus we have a graph $\Gamma$ on the vertex
set $X_{0}\cup X_{1}\cup X_{2};$ and then we construct a graph $ab$ on the
subset of vertices $X_{0}\cup X_{2}$ by declaring, for $i,j\in X_{0}\cup
X_{2},$ that $\left\{  i,j\right\}  $ is an edge of $ab$ if and only if there
is a path in $\Gamma$ from $i$ to $j$. Note $ab$ is a union of cliques, and
each path may be taken, without loss of generality, to have edges taken
alternately from $a$ and $b,$ if null initial and terminal edges are allowed.

The relevance of $\mathcal{P}_{X}$ in the present paper lies in its
relationship with the semigroup of binary relations. It takes just a short
paragraph to explain how the partition monoid $\mathcal{P}_{X}$ is
conceptually dual to the semigroup of binary relations; we begin with a
categorical description of $\mathcal{B}_{X}.$ A \emph{relation from }$X$\emph{
to }$Y$ over a category with products is a subobject of the product $X\times
Y.$ If the base category also has pullbacks and images, there is a natural way
of defining the product of a relation from $X$ to $Y$ with a relation from $Y$
to $Z.$ Over the category $\mathbf{Set}$ this product is the usual composition
of relations, and so is associative; thus we derive the category of binary
relations over $\mathbf{Set},$ and the endomorphism monoid at the object $X$
is simply $\mathcal{B}_{X}.$ We could call $\mathcal{B}_{X}$ the
\emph{relation monoid at }$X$ over $\mathbf{Set}.$ Now the category
$\mathbf{Set}^{\text{opp}}$ also turns out (non-trivially) to satisfy the
conditions above. A product in $\mathbf{Set}^{\text{opp}}$ is the same as a
coproduct of sets, and a subobject in $\mathbf{Set}^{\text{opp}}$ is a
quotient object in $\mathbf{Set}.$ So elements of the relation monoid at $X$
over $\mathbf{Set}^{\text{opp}}$ do indeed constitute the partition monoid
$\mathcal{P}_{X}$ as defined above.

More significantly, the multiplication operations in $\mathcal{P}_{X}$ and
$\mathcal{B}_{X}$ are categorical dual constructs as well. The monoids
$\mathcal{B}_{X}$ and $\mathcal{P}_{X}$ share the properties of containing, in
a canonical way, the semigroups $\mathcal{T}_{X}$ of all maps $X\rightarrow
X,$ and $\mathcal{I}_{X}$ of all partial injective maps $X\rightarrow X;$
$\mathcal{P}_{X}$ additionally has the property of being a regular $\ast
$-semigroup, and containing, in a canonical way, the dual symmetric inverse
monoid $\mathcal{I}_{X}^{\ast}.$ Importantly for present purposes,
$\mathcal{P}_{X}$ also carries, like $\mathcal{B}_{X},$ an inclusion order, in
which we write (for $a,b\in\mathcal{P}_{X}$) $a\subseteq b$ if and only if
every edge of $a$ is an edge of $b.$ This order is the usual refinement order
on set partitions, and is compatible with the product in $\mathcal{P}_{X}$:

\begin{lemma}
Let $a,b,c\in\mathcal{P}_{X}$ with $a\subseteq b.$ Then $ac\subseteq bc$ and
dually $ca\subseteq cb.$
\end{lemma}

\begin{proof}
Let $i,\dots j$ be a path in $ac;$ each edge from $a$ is an edge of $b,$ by
hypothesis. So $i,\dots j$ is a path in $bc.$
\end{proof}

In the canonical copy of $\mathcal{I}_{X}$ in $\mathcal{P}_{X},$the natural
partial order agrees with inclusion, and in the canonical copy of
$\mathcal{I}_{X}^{\ast},$ the natural partial order agrees with reverse
inclusion. It follows, as with $\mathcal{B}_{X},$ that $\leq$ and $\subseteq$
are related in a complex way, which we may seek to understand in the same way
as was used above for $\mathcal{B}_{X}.$ In particular, we see again that
$\leq\cap\subseteq$ and $\leq\cap\supseteq$ are non-trivial on $\mathcal{P}%
_{X}$ but have a trivial conjunction; and likewise that there can be no
partial order on $\mathcal{P}_{X}$ which contains both $\leq$ and $\subseteq,$
or both $\leq$ and $\supseteq.$

Just as in section 5, composites help identify joins of the orders within the
lattice of preorders on $\mathcal{P}_{X}.$ We need the following definitions:
let $\mathbf{d}$ be the trivial partition over $X,$ thus corresponding to the
discrete graph on $X\sqcup X,$ and let $\mathbf{k}$ be the universal partition
over $X,$ thus corresponding to the complete graph on $X\sqcup X.$ Clearly,
$\mathbf{d}$ is the least element, and $\mathbf{k}$ the greatest, in the
inclusion order of $\mathcal{P}_{X}.$

\begin{lemma}
For all $a\in\mathcal{P}_{X},$ \newline(i)~$~~\mathbf{d}a\mathbf{d}%
=\mathbf{d};$\newline(ii)~~if $a$ has a transversal block, then $\mathbf{k}%
a\mathbf{k}=\mathbf{k};$\newline(iii)~if $a$ has no transversal block, then
$a\mathbf{k}a=a.$
\end{lemma}

\begin{proof}
By direct computation.
\end{proof}

\begin{corollary}
For all $a\in\mathcal{P}_{X},$ $a\mathbf{d},\mathbf{d}a,a\mathbf{k}%
,\mathbf{k}a$ are idempotents.
\end{corollary}

\begin{proposition}
For all $a,b\in\mathcal{P}_{X},$ $a\supseteq\circ\leq b$ if and only if
$a\mathbf{d}a\supseteq b\mathbf{d}b.$
\end{proposition}

\begin{proof}
Suppose there is $c\in\mathcal{P}_{X}$ such that $a\supseteq c\leq b.$ Then
$c=bx=yb$ for some $x,y\in\mathcal{P}_{X},$ and we have $a\mathbf{d}a\supseteq
c\mathbf{d}c=by\mathbf{d}xb\supseteq b\mathbf{d}^{3}b=b\mathbf{d}b.$

For the converse, we see that $a\mathbf{d}a\supseteq b\mathbf{d}b$ implies
$a=aa^{\ast}a\supseteq a\mathbf{d}a\supseteq b\mathbf{d}b\leq b,$ the latter
following from idempotency of $b\mathbf{d}$ and $\mathbf{d}b$.
\end{proof}

\begin{corollary}
In $\mathcal{P}_{X},$ \newline(i)~~~$\supseteq\circ\leq$ is a
preorder;\newline(ii)~~$\leq\circ\supseteq$ is contained in $\supseteq
\circ\leq;$\newline(iii)~$\supseteq\circ\leq$ is the join of $\supseteq$ and
$\leq$ in the lattice of preorders.
\end{corollary}

\begin{proposition}
For all $a,b\in\mathcal{P}_{X},$ $a\subseteq\circ\leq b$ if and only if
$a\mathbf{k}a\subseteq b\mathbf{k}b.$
\end{proposition}

\begin{proof}
Suppose there is $c\in\mathcal{P}_{X}$ such that $a\subseteq c\leq b.$ Then
there are $x,y\in\mathcal{P}_{X}$ such that $c=bx=yb$ and hence $a\mathbf{k}%
a\subseteq c\mathbf{k}c=by\mathbf{k}xb\subseteq b\mathbf{k}^{3}b=b\mathbf{k}%
b.$

Conversely, the condition implies $a=aa^{\ast}a\subseteq a\mathbf{k}a\subseteq
b\mathbf{k}b\leq b,$ the latter following from idempotency of $b\mathbf{k}$
and $\mathbf{k}b$.
\end{proof}

\begin{corollary}
In $\mathcal{P}_{X},$ \newline(i)~~~$\subseteq\circ\leq$ is a
preorder;\newline(ii)~~$\leq\circ\subseteq$ is contained in $\subseteq
\circ\leq;$\newline(iii)~$\subseteq\circ\leq$ is the join of $\subseteq$ and
$\leq$ in the lattice of preorders.
\end{corollary}

As in the previous section, this information is summarised in Fig.2, a Hasse
diagram depicting the sublattice of the lattice of preorders on $\mathcal{P}%
_{X}$ generated by $\leq,\subseteq,$ and $\supseteq$~. Filled circles denote
orders, and parentheses contain conditions for $a$ and $b$ to be related in
the corresponding preorder.%

\begin{center}
\includegraphics[
trim=0.000000in 0.000000in 0.000738in -0.002311in,
natheight=3.300900in,
natwidth=3.691100in,
height=7.0348cm,
width=7.8509cm
]%
{Fig2.tif}%
\label{F2}%
\end{center}

\begin{center}
{\small Fig. 2. A sublattice of preorders on }$\mathcal{P}_{X}.$%
{\small ~Filled circles denote orders.}
\end{center}

\end{document}